\documentclass[11pt]{amsart} 

\usepackage{amsfonts, amsmath, lscape}

\usepackage[utf8]{inputenc}
\usepackage{comment}
\usepackage[T1]{fontenc}
\usepackage[english]{babel}
\usepackage{enumitem}
\usepackage{amsthm}
\usepackage{amssymb}
\usepackage{hyperref}
\usepackage{tikz}
\usepackage{tikz-network}
\usetikzlibrary{shapes}
\usetikzlibrary{positioning}
\tikzstyle{vertex}=[circle,draw,inner sep=0pt,minimum size=6pt]
\newcommand{\vertex}{\node[vertex]}
\usepackage{csquotes}
\usepackage{listings}
\usepackage{float, lscape}

\newtheorem{theorem}{Theorem}[section]
\newtheorem{lemma}{Lemma}[section]
\newtheorem{proposition}{Proposition}[section]

\usepackage{caption} 

\begin{document}

\title[Rank Augmentation Theorem from rank $3$ to rank $4$]{A rank augmentation theorem for rank $3$ string C-group representations of the symmetric groups}

\author{Julie De Saedeleer}
\address{Julie De Saedeleer, Département de Mathématique, Université libre de Bruxelles, C.P.216, Boulevard du Triomphe, 1050 Brussels, Belgium}
\email{julie.de.saedeleer@ulb.be}
\author{Dimitri Leemans}
\address{Dimitri Leemans (corresponding author), Département de Mathématique, Université libre de Bruxelles, C.P.216, Boulevard du Triomphe, 1050 Brussels, Belgium}
\email{leemans.dimitri@ulb.be}
\author{Jessica Mulpas}
\address{Jessica Mulpas, Département de Mathématique, Université libre de Bruxelles, C.P.216, Boulevard du Triomphe, 1050 Brussels, Belgium}
\email{jessica.mulpas@ulb.be}
\maketitle

\begin{abstract}
We give a rank augmentation technique for rank $3$ string C-group representations of the symmetric group $S_n$ and list the hypotheses under which it yields a valid string C-group representation of rank $4$ thereof.
\end{abstract}

\maketitle

\section{Introduction}

In the Atlas of abstract regular polytopes for small almost simple groups ~\cite{AtlasSmallGroups}, Leemans and Vauthier have made all regular polytopes whose automorphism group is a symmetric group of degree between $5$ and $9$ available. In 2011, Fernandes and Leemans, inspired by the experimental data above, proved three theorems about string C-group representations for $S_n$.
Among other things, they showed that $S_n$ has string C-group representations of rank $r$ for all $r\in\{3, \ldots, n-1\}$ (see \cite[Theorem 3]{PolSymGroups}). 
Recently, Cameron, Fernandes and Leemans showed that, for a given $n$ sufficiently large, the number of string C-group representations of rank $r$ for $S_n$ with $r\geq (n+3)/2$ is equal to the number of string C-group representations of rank $r+1$ for $S_{n+1}$~\cite{CFL2022}. Hence when $n$ and $r$ are large enough, a lot is known about string C-group representations of rank $r$ for $S_n$.

Similar results were obtained for the alternating groups. Cameron, Fernandes, Leemans and Mixer showed that the highest possible rank for a string C-group representation of $A_n$ is $\lfloor (n-1)/2\rfloor$ when $n\geq 12$~\cite{altn}. 
More recently, Fernandes and Leemans showed that, if $n\geq 12$, the group $A_n$ has a string C-group representation of rank $r$ for all $r\in\{3, \ldots, \lfloor (n-1)/2\rfloor\}$~\cite{altnallranks}. 

Note also, for sake of completeness that, in 2013, Leemans and Kiefer counted all pairs of commuting involutions in $A_n$ and $S_n$ up to isomorphism. These pairs $\{\rho_0,\rho_2\}$ may be extended to rank $3$ string C-group representations $(S_n,\{\rho_0,\rho_1,\rho_2\})$. That was a first step in the classification of rank $3$ string C-groups for $S_n$, and, to the authors' opinion, obtaining such a classification is hopeless at present.

The results of Fernandes and Leemans~\cite{PolSymGroups,altnallranks} led Brooksbank and Leemans, in 2019, to look for and find a sufficient condition under which a string C-group representation of rank $r$ for a group $G$ yields a rank $r-1$ string C-group representation for this same group (see~\cite[Theorem 1.1]{RRT}).
This paper takes the natural first step into reversing the process of Brooksbank and Leemans for symmetric groups and developing a rank augmentation technique, i.e. giving a sufficient condition for a rank $3$ string C-group representation for $S_n$ to yield a rank $4$ one. As the reader will see, augmenting the rank of a string C-group is a much more complicated task than reducing it.

\begin{table}
\begin{tabular}{|c|c|c|c|c|c|c|c|c|c|c|c|c||}
\hline
$G$&Rk 3&Rk 4&RRT&RAT\\
\hline 
$S_5$&4&1&1&1\\
$S_6$&2&4&2&2\\
$S_7$&35&7&6&6\\
$S_8$&68&36&13&12\\
$S_9$&129&37&20&15\\
$S_{10}$&413 &203 &69&49\\
$S_{11}$&1221 &189 &70&70\\
\hline
\end{tabular}
\caption{The impact of the rank reduction theorem in rank four and a tentative rank augmentation theorem on $S_n$ ($5\leq n\leq 11$).}\label{sn}
\end{table}

In Table~\ref{sn}, we give the numbers obtained from the following computer experiment:
we took the string C-group representations of rank 3 and 4 of $S_n$ with $5\leq n \leq 11$ and compared the number of rank $4$ ones that can be reduced in rank according to Leemans and Brooksbank to the number that can be obtained by using the technique we present here. The first and second columns give the number of representations of rank $3$ and $4$ for $S_n$ up to isomorphism and duality.
For each representation of rank 4, we get a 4-tuple of generators $(\rho_0,\rho_1,\rho_2,\rho_3)$.
In the third column (named RRT), we give the number of rank $4$ string C-group representations of $S_n$ one can start from to get a rank $3$ one using the Rank Reduction Theorem.
In the last column, we enumerated those string C-group representations of rank 4 that have a transposition (i.e. an element of $S_n$ that only swaps two points and fixes all the others) as first or last generator. In the latter case, the dual has a transposition as first generator. We only count each one when reducing its rank according to Leemans and Brooksbank's technique gives one of the rank $3$ representations.
This motivates our main Theorem (see Theorem~\ref{RAT1}). Indeed, we decide to study what happens if we suppose that we start from a rank $3$ string C-group representation $(G,\{\rho_0,\rho_1,\rho_2\})$ of $G\cong S_n$ and we split the involution $\rho_1$ in two, taking one of its transpositions as $\rho_{-1}$ and removing that transposition from $\rho_1$. Such a rank augmentation procedure will thus always yield four generators, the first one being a transposition.

The paper is organised as follows.
In Section~\ref{prelim}, we give the basic notions needed to understand this paper.
In Section~\ref{rrt}, we recall the Rank Reduction Theorem of Brooksbank and Leemans mentioned above.
In Section~\ref{RRT3-4}, we state our Rank Augmentation Theorem (Theorem~\ref{RAT1}) for string C-groups of rank three.
In Section~\ref{proof}, we prove Theorem~\ref{RAT1}.
In Section~\ref{conclusion}, we conclude this paper by giving some examples motivating the extra hypotheses of Theorem~\ref{RAT1} and we suggest ways to go further in the exploration of rank augmentation.

\section{Preliminaries}\label{prelim}
\subsection{String C-groups}

A {\em string C-group representation} (or string C-group for short) is a pair $(G,S)$ with $S:=\{\rho_0,\rho_1,...,\rho_{r-1}\}$ a generating set of involutions of the group $G$, satisfying the following two properties.
\begin{description}[style=unboxed,leftmargin=0cm]
\item[(SP)] 
the {\em string property}, that is $(\rho_i\rho_j)^2=1_G$ for all $i,j \in \{ 0,1,...,r-1 \}$ with $\mid i-j \mid \geq 2$;
\item[(IP)] the {\em intersection property}, that is $\langle \rho_i \mid i \in I \rangle \cap \langle \rho_j \mid j \in J \rangle = \langle \rho_k \mid k \in I \cap J \rangle$ for any $I, J \subseteq \{ 0, 1,...,r-1 \}$.
\end{description}
When $(G,S)$ only satisfies the string property it is called a {\em string group generated by involutions} (or {\em sggi} for short). The {\em rank} of $(G,S)$ is the size of $S$.

Note that, from the definition above, one can  observe that string C-groups are smooth quotients of Coxeter groups with string diagrams. It is also a well-known fact that string C-groups are in one-to-one correspondence with  abstract regular polytopes, the latter being equivalent geometric formulations of the former ~\cite{ARP}.

For any subset $I\subseteq \{0, \ldots, r-1\}$, we denote $\langle \rho_j : j \in \{0, \ldots, r-1\}\setminus I\rangle$ by $G_I$. When $G$ is a string C-group, so is $G_I$. If $I = \{i\}$, we denote $G_{\{i\}}$ by $G_i$. Similarly, if $I = \{i,j\}$, we denote $G_{\{i,j\}}$ by $G_{i,j}$.


The {\em Schl\"afli type} of a string C-group representation $(G,\{\rho_0, \ldots, \rho_{r-1}\})$ is the ordered set $\{p_1,p_2,...,p_{r-1} \}$ where $p_i$ is the order of the element $\rho_{i-1}\rho_{i}$ for $i=1, \ldots, r-1$.


The {\em dual} of a string C-group representation $(G,\{\rho_0,\rho_1,...,\rho_{r-1}\})$ is the string C-group representation $(G,\{\rho_{r-1},\rho_{r-2},...,\rho_{0}\})$.


As in~\cite{RankofPolAltGroups}, for a sggi $(G, \{\rho_0,\rho_1,...,\rho_{r-1}\})$, an involution $\tau$ in a supergroup of $G$ such that $\tau\notin G$ and a fixed $k\in\{0,1,...,r-1\}$, one can define a new sggi $(G^*,S^*)$ where $S^* := \{\rho_i \tau^{\delta_{i,k}}\vert i\in\{0,1,...,r-1\}\}$ and $G^* := \langle S^*\rangle$
that we call the {\em sesqui-extension} of $G$ with respect to $\rho_k$ and $\tau$ (or {\em $k$-sesqui-extension} of $G$ with respect to $\tau$).
We have the following result.

\begin{proposition} \cite[Lemma 5.4]{RankofPolAltGroups}\label{sesqui}
If $G=\langle \rho_0,\rho_1,...,\rho_{r-1}\rangle$ and $G^*=\langle \rho_i \tau^{\delta_{i,k}}\vert i\in\{0,1,...,r-1\}\rangle$ is a $k$-sesqui extension of $G$ then $G^*\cong G$ or $G^*\cong G\times C_2$ (when $\tau\in G^*$).
\end{proposition}

\subsection{Permutation representation graphs and CPR graphs}

Let $(G, \Omega)$ (with $\Omega := \{1,...,n\}$) be a permutation group generated by $r$ involutions $\rho_0,\rho_1,...,\rho_{r-1}$.
We define the {\em permutation representation graph} of $G$ to be the edge-labelled undirected multigraph $\mathcal{G}$ with $\Omega$ as vertex set and an edge with label $i$ between vertices $a$ and $b$ whenever $a\neq b$ and $\rho_i(a)=b$ (in which case, of course, $\rho_i(b)=a$).

When $(G,\{\rho_0,\rho_1,...,\rho_{r-1}\})$
is a string C-group, $\mathcal{G}$ is called a {\em CPR graph} as in~\cite{cprbible}. 

The following lemma gives the possible shapes of connected components of subgraphs of a CPR graph when looking at pairs of labels that are not consecutive.

\begin{lemma}\cite[Proposition 3.5]{cprbible}\label{cprlem}
Each connected component of the subgraph of a CPR graph $\mathcal{G}$ induced by edges of labels $i$ and $j$ for $\vert i-j\vert \geq 2$ (i.e. by edges of labels $i$ and $j$ where $\rho_i$ and $\rho_j$ commute) is either a single vertex, a single edge, a double edge or an alternating square.
\end{lemma}

Observe that the proof of this lemma only requires $\mathcal G$ to be the permutation representation graph of an sggi.
Trivially, two generators commute unless they have adjacent corresponding edges.

Let us also notice that since $S_n$ is a transitive group in its natural action, any representation of $S_n \curvearrowright \{1,...,n\}$ by a permutation representation graph will be connected. Here we use $G\curvearrowright S$ to denote the action of a group $G$ on a set $S$.

Let $\mathcal{G}$ be the permutation representation of a string group generated by involutions $(G,\{\rho_0,\rho_1,...,\rho_{r-1}\})$ seen as a permutation group acting on a set $\Omega:=\{1, \ldots, n\}$. 
For $i,j\in\{0,1,...,r-1\}$ we define $\mathcal{G}_{i,j}$ to be the subgraph of $\mathcal{G}$ with vertex set $\Omega$ and all the edges in $\mathcal{G}$ except the ones of labels $i$ and $j$.

With the notations introduced in the previous section, we observe that $\mathcal{G}_{i,j}$ is the permutation representation graph for $G_{i,j}$ and that the vertices in its connected components form the orbits of $G_{i,j}\curvearrowright \Omega$.


In sections \ref{RRT3-4} and \ref{proof}, we will make use of some extra features when representing permutation representation graphs in order to clarify our arguments:
\begin{itemize}
\item A vertex, usually represented as $\circ$ will be filled in black when we represent a subgraph of a permutation representation graph and wish to emphasize that this particular vertex has no other incident edge than the ones appearing on the picture;
\item An edge will be crossed out when we wish to emphasize that its appearance leads to a contradiction.
\end{itemize}
\section{The Rank Reduction Theorem}\label{rrt}
In \cite{PolSymGroups}, Fernandes and Leemans have shown that when $n\geq 4$, the group $S_n$ has string C-group representations of rank $r$ for each $3\leq r \leq n-1$.
In~\cite{altn}, Cameron, Fernandes, Leemans and Mixer proved that, when $n\geq 12$, the maximum rank of a string C-group representation of $A_n$ is $\lfloor (n-1)/2 \rfloor$. Later on, Fernandes and Leemans proved in~\cite{altnallranks} that, when $n\geq 12$, the group $A_n$ has string C-group representations of rank $r$ for each $3\leq r \leq \lfloor (n-1)/2\rfloor$.
In order to do so, they used a rank reduction technique on known string C-group representations for $S_n$ and $A_n$. 
Brooksbank and Leemans then generalized this technique and discovered that it does not depend on the particular settings of symmetric and alternating groups~\cite{RRT}.

\begin{theorem}(RRT)~\cite[Corollary 1.2]{RRT}
Let $(G,\{\rho_0, \rho_1,...,\rho_{r-1}\})$ be a string C-group of rank $r \geq 4$ such that no two adjacent generators commute. If $\rho_2\rho_3$ has odd order then $(G,\{ \rho_1,\rho_0\rho_2,\rho_3,...,\rho_{r-1}\})$ is a string C-group of rank $r-1$.
\end{theorem}

Note that, in the main theorem of \cite{RRT}, the sufficient condition for a string C-group representation $(G,\{\rho_0, \rho_1,...,\rho_{r-1}\})$ of rank $r$ to yield a representation of rank $r-1$ for $G$ as in the theorem above is for $\rho_0$ to belong to $\langle\rho_0\rho_2,\rho_3 \rangle$. The above theorem, that we call Rank Reduction Theorem (or RRT) here, is presented as a Corollary to this main theorem.

\section{The Rank Augmentation Theorem}\label{RRT3-4}

Suppose that we obtain a string C-group representation of rank $r-1\geq 3$ for a group $G$, say $(G,\{g_0, \tilde g_1,...,g_{r-2}\})$ by applying the RRT as stated in the previous section to a string C-group representation of rank $r\geq 4$ $(G,\{\rho_0,\rho_1,\rho_2,\rho_3,...,\rho_{r-1}\})$. Then $g_0=\rho_1$, $\tilde g_1=\rho_0\rho_2$ and $g_i=\rho_{i+1}$ for any $i\in\{2,...,r-2\}$.
If we wish to reverse the operation, we need to start with $(G,\{g_0,\tilde g_1,g_2,...,g_{r-2}\})$, find two involutions $a$ and $b$ such that $\tilde g_1=ab$ and create a new string of generators by placing $a$ in front of $g_0$ and $b$ between $g_0$ and $g_2$, in place of $\tilde g_1$. Retrieving our usual indexed notations, $a$ takes the role of a $g_{-1}$ while $b$ takes the role of a new $g_1$.

In terms of the permutation representation graph, this amounts to replace some $1$-edges by $(-1)$-edges.

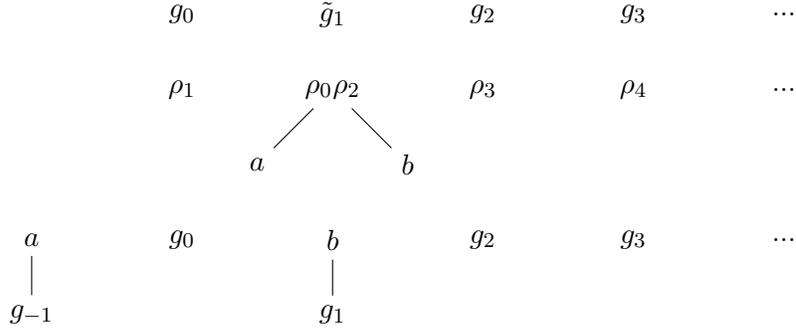
\begin{center}
\begin{tikzpicture}
\node[draw=none,fill=none] (1) at (2,3) {$\rho_1$};
\node[draw=none,fill=none] (2) at (4,3) {$\rho_0\rho_2$};
\node[draw=none,fill=none] (3) at (6,3) {$\rho_3$};
\node[draw=none,fill=none] (4) at (8,3) {$\rho_4$};
\node[draw=none,fill=none] (5) at (10,3) {...};
\node[draw=none,fill=none] (6) at (3,2) {$a$};
\node[draw=none,fill=none] (7) at (5,2) {$b$};
\draw[-] (2) -- (6);
\draw[-] (2) -- (7);

\node[draw=none,fill=none] (8) at (2,4) {$g_0$};
\node[draw=none,fill=none] (9) at (4,4) {$\tilde g_1$};
\node[draw=none,fill=none] (10) at (6,4) {$g_2$};
\node[draw=none,fill=none] (11) at (8,4) {$g_3$};
\node[draw=none,fill=none] (12) at (10,4) {...};

\node[draw=none,fill=none] (13) at (0,1) {$a$};
\node[draw=none,fill=none] (14) at (2,1) {$g_0$};
\node[draw=none,fill=none] (15) at (4,1) {$b$};
\node[draw=none,fill=none] (16) at (6,1) {$g_2$};
\node[draw=none,fill=none] (17) at (8,1) {$g_3$};
\node[draw=none,fill=none] (18) at (10,1) {...};

\node[draw=none,fill=none] (19) at (0,0) {$g_{-1}$};
\node[draw=none,fill=none] (20) at (4,0) {$g_1$};

\draw[-] (13) -- (19);
\draw[-] (15) -- (20);

\end{tikzpicture}
\end{center}


Note that, if $\{g_0,\tilde g_1,g_2,...,g_{r-2}\}$ generates $S_n$ for some $n\in\mathbb{N}_0$, then so does $\{g_{-1},g_0,g_1,g_2,...,g_{r-2}\}$ since $S_n$ is the largest group acting on $n$ elements and $g_{-1}g_1=\tilde g_1$ so that $\langle g_0,\tilde g_1,g_2,...,g_{r-2}\rangle\leq \langle g_{-1},g_0,g_1,g_2,...,g_{r-2}\rangle$.

A natural first step into determining conditions under which applying such modifications to a string  C-group yields a string C-group is to treat the case where $g_{-1}$ is a transposition.

Within this framework, we prove the following rank augmentation theorem for small rank string C-group representations of $S_n$.



\begin{theorem}\label{RAT1}
Let $n\geq 5$ and $\Gamma\cong S_n$ be the natural permutation representation of $S_n$ on $n$ points. 
Let $(\Gamma,\{\rho_0,\tilde{\rho}_1,\rho_2\})$ be a rank $3$ string C-group representation of the symmetric group $\Gamma$ and let $\mathcal{G}$ be its CPR graph. 
Suppose that $\tilde{\rho}_1$ is not a transposition. 
Let $\mathcal{G}'$ be the permutation representation graph obtained by replacing, in $\mathcal{G}$, a $1$-edge that is not adjacent to any $2$-edge by a $(-1)$-edge. Let $(\Gamma,\{\rho_{-1},\rho_0,\rho_1,\rho_2\})$ be the permutation representation corresponding to $\mathcal{G}'$ (where $\rho_{-1}\rho_1=\tilde{\rho}_1$). 
Suppose that the $\Gamma_{-1,2}$-orbits are of size at most $3$ and suppose moreover that one of the following is satisfied:
\begin{enumerate}
\item $\Gamma_2$ has an orbit of size $4$ and $\Gamma_{-1}$ acts imprimitively on the connected components of  $\mathcal{G}_{-1}$ that contain more than one $\mathcal{G}_{-1,2}$-connected component on three vertices;
\item $\Gamma_2$ has an orbit of size $5$, $\rho_0$, $\rho_1$ and $\rho_2$ are all even permutations and $\Gamma_{-1}$ acts imprimitively on 
the connected components of  $\mathcal{G}_{-1}$ that contain more than one $\mathcal{G}_{-1,2}$-connected component on three vertices;
\item $\Gamma_2$ has an orbit  of size $6$.
\end{enumerate}

Then $(\Gamma,\{\rho_{-1},\rho_0,\rho_1,\rho_2\})$ is a string C-group representation for $S_n$. 
\end{theorem}

As previously mentioned, it is clear that $\{\rho_{-1},\rho_0,\rho_1,\rho_2\}$ generates the whole of $S_n$.

Let us also mention that choosing the $(-1)$-edge of $\mathcal{G}'$ non-adjacent to any $2$-edge is necessary to ensure the commuting property is not broken.

Now note that, since $\Gamma_{-1,2}$ has no orbit of size more than $3$, $\Gamma_2$ can only have one of size superior to $3$ and it must be the one containing the two vertices switched by $\rho_{-1}$.  In terms of connected components in subgraphs of $\mathcal{G}'$, this is equivalent to saying that, since $\mathcal{G}'_{-1,2}$ has no connected component with more than $3$ vertices, the only connected component of $\mathcal{G}'_2$ with more than $3$ vertices is the one containing the only $(-1)$-edge in $\mathcal{G}'$.
This aforementioned connected component can either be 

\begin{center}
\begin{tikzpicture}
\vertex[fill=black] (1) at (0,0) {};
\vertex[] (2) at (1,0) {};
\vertex[] (3) at (2,0) {};
\vertex[] (4) at (3,0) {};
\draw[-] (1) -- (2)
 node[pos=0.5,above] {-1};
\draw[-] (3) -- (2)
 node[pos=0.5,above] {0};
\draw[-] (3) -- (4)
 node[pos=0.5,above] {1};
\end{tikzpicture}
\end{center}

\noindent or 

\begin{center}
\begin{tikzpicture}
\vertex[fill=black] (1) at (0,0) {};
\vertex[] (2) at (1,0) {};
\vertex[] (3) at (2,0) {};
\vertex[] (4) at (3,0) {};
\vertex[] (5) at (4,0) {};
\draw[-] (1) -- (2)
 node[pos=0.5,above] {0};
\draw[-] (3) -- (2)
 node[pos=0.5,above] {-1};
\draw[-] (3) -- (4)
 node[pos=0.5,above] {0};
\draw[-] (4) -- (5)
 node[pos=0.5,above] {1};
\end{tikzpicture}
\end{center}

\noindent or finally
 
\begin{center}
\begin{tikzpicture}
\vertex[] (1) at (0,0) {};
\vertex[] (2) at (1,0) {};
\vertex[] (3) at (2,0) {};
\vertex[] (4) at (3,0) {};
\vertex[] (5) at (4,0) {};
\vertex[] (6) at (5,0) {};
\draw[-] (1) -- (2)
 node[pos=0.5,above] {1};
\draw[-] (3) -- (2)
 node[pos=0.5,above] {0};
\draw[-] (3) -- (4)
 node[pos=0.5,above] {-1};
\draw[-] (4) -- (5)
 node[pos=0.5,above] {0};
\draw[-] (5) -- (6)
 node[pos=0.5,above] {1};
\end{tikzpicture}
\end{center}


This is due to the fact that the $(-1)$-edge can only be adjacent to $0$-edges and these $0$-edges can only be adjacent to extra $1$-edges as a $2$-edge would have to appear in a $0,2$-alternating square that would have a second $2$-edge, adjacent to the $(-1)$-edge, as represented below, a contradiction to our choice of $(-1)$-edge.

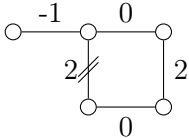
\begin{center}
\begin{tikzpicture}
\vertex[] (1) at (0,1) {};
\vertex[] (2) at (1,1) {};
\vertex[] (3) at (2,1) {};
\vertex[] (4) at (2,0) {};
\vertex[] (5) at (1,0) {};
\draw[-] (1) -- (2)
 node[pos=0.5,above] {-1};
\draw[-] (3) -- (2)
 node[pos=0.5,above] {0};
 
\draw[-] (3) -- (4)
 node[pos=0.5,right] {2};
\draw[-] (2) -- (5)
 node[pos=0.5,left] {2}
 node[draw,strike out,pos=0.45] {}
 node[draw,strike out,pos=0.55] {};
\draw[-] (5) -- (4)
 node[pos=0.5,below] {0};
\end{tikzpicture}
\end{center}


Note that the three possibilities we have just highlighted also show that $\mathcal{G}'_{-1,2}$ has at least one connected component on three vertices.

One can also note that if $\mathcal{G}'$ has no such subgraph, $\rho_0$ and $\rho_1$ commute.
But then $\Gamma\cong\langle\rho_{-1},\rho_0\rangle \times \langle\rho_1,\rho_2\rangle$, so $\Gamma$ cannot be a symmetric group, a contradiction.
For the same reason, no two consecutive generators in $(\rho_{-1},\rho_0,\rho_1,\rho_2)$ can commute.

\section{The proof of Theorem \ref{RAT1}} \label{proof}

Let us recall the following proposition.
\begin{proposition}\cite[Proposition 2E16(a)]{ARP}\label{breakinparts}
A string group generated by involutions $(G,\{\rho_0,...,\rho_{r-1}\})$ is a string C-group (i.e. satisfies the intersection property) if and only if 
\noindent  \begin{itemize}
 \item[$\bullet$] $(G_{0},\{\rho_1,...,\rho_{r-1}\})$ is a string C-group;
 \item[$\bullet$] $(G_{r-1},\{\rho_0,...,\rho_{r-2}\})$ is a string C-group, and
 \item[$\bullet$] $G_{0}\cap G_{r-1} = G_{0,r-1}:=\langle\rho_1,...,\rho_{r-2}\rangle$
\end{itemize}

\end{proposition}

We will also make use of the following lemma which is folklore.

\begin{lemma} \label{dihedral}
Any dihedral group $G$ generated by two distinct involutions $\rho_0$ and $\rho_1$ has a string C-group representation $(G,\{\rho_0,\rho_1\})$.
\end{lemma}

\begin{proof}
Straightforward.\end{proof}

In order to prove that $(\Gamma,\{\rho_{-1},\rho_0,\rho_1,\rho_2\})$, the group associated to $\mathcal{G}'$, is a string C-group we first observe that $(\Gamma,\{\rho_{-1},\rho_0,\rho_1,\rho_2\})$ is a string group generated by involutions. Indeed, since we have chosen our $(-1)$-edge non-adjacent to any $2$-edge and in place of a $1$-edge, the only $(-1)$-edge in $\mathcal{G}'$ is neither adjacent to any $2$-edge nor to any  $1$-edge. Hence $\rho_{-1}$ commutes with both $\rho_1$ and $\rho_2$. Since $\rho_0$ and $\rho_2$ are unchanged, they still commute with one another.

Due to Proposition \ref{breakinparts}, we now only need to check that
\begin{itemize}
    \item $(\Gamma_1,\{\rho_0,\rho_1,\rho_2\})$ is a string C-group
    \item $(\Gamma_2,\{\rho_{-1},\rho_0,\rho_1\})$ is a string C-group
    \item $\Gamma_{-1} \cap \Gamma_{2}=\Gamma_{-1,2}$
\end{itemize}

We have shown in the previous section that $\mathcal{G}'$ must have a subgraph of the shape

\begin{center}
\begin{tikzpicture}
\vertex[] (1) at (0,0) {};
\vertex[] (2) at (1,0) {};
\vertex[] (3) at (2,0) {};
\draw[-] (1) -- (2)
node[pos=0.5,above]{0};
\draw[-] (2) -- (3)
node[pos=0.5,above] {1};
\end{tikzpicture}
\end{center}

that is, $\mathcal{G}'_{-1,2}$ must have a connected component on three vertices or, in other words, $\Gamma_{-1,2}$ must have at least one orbit of size $3$.

In fact, the action of $\Gamma_{-1,2}$ on the vertices of $\mathcal{G}'$ is the one of a dihedral group $D_6$ or $D_{12}$:


\begin{lemma}\label{-1,2}
Let $(G,\{\rho_{0},\rho_1\})$ be a string C-group representation of rank two.
Let $\mathcal{G}$ be the permutation representation graph of $G$.
If $G$  has at least one orbit of size $3$ and none of greater size then the potential $\mathcal{G}$-connected components are

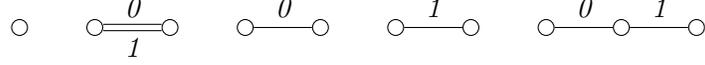
\begin{center}
\begin{tikzpicture}
\vertex[] (1) at (1,0) {};

\vertex[] (2) at (2,0) {};
\vertex[] (3) at (3,0) {};
\draw[thin,double distance=2pt] (2) -- (3)
 node[pos=0.5,above] {0} node[pos=0.5,below] {1};

\vertex[] (4) at (4,0) {};
\vertex[] (5) at (5,0) {};
\draw[-] (4) -- (5)
 node[pos=0.5,above] {0};

\vertex[] (6) at (6,0) {};
\vertex[] (7) at (7,0) {};
\draw[-] (6) -- (7)
 node[pos=0.5,above] {1};

\vertex[] (8) at (8,0) {};
\vertex[] (9) at (9,0) {};
\vertex[] (10) at (10,0) {};
\draw[-] (8) -- (9)
 node[pos=0.5,above] {0};
\draw[-] (9) -- (10)
 node[pos=0.5,above] {1};
\end{tikzpicture}
\end{center}

\noindent and $G$ is isomorphic to either $S_3\cong D_6$ or $C_2\times S_3\cong D_{12}$.
\end{lemma}
\begin{proof}
Straightforward.
\end{proof}

\subsection{Proof that $\Gamma_{-1}$ is a string C-group}\label{part1}

As a subgroup of $\Gamma$, $\Gamma_{-1}$ is a string group generated by  involutions.

\noindent According to Proposition \ref{breakinparts}, it is enough to show that 
\begin{itemize}
    \item $(\Gamma_{-1,0},\{\rho_1,\rho_2\})$ is a string C-group
    \item $(\Gamma_{-1,2},\{\rho_0,\rho_1\})$
    is a string  C-group
    \item $\Gamma_{-1,0}\cap\Gamma_{-1,2}=\Gamma_{-1,0,2}:=\langle \rho_1\rangle$
\end{itemize}

The first two points are clear by Lemma \ref{dihedral}.
It remains to show that $\Gamma_{-1,0}\cap\Gamma_{-1,2}=\langle \rho_1\rangle$.

The subgroup $\Gamma_{-1,2}$ is either $D_6$ or $D_{12}$ by Lemma~\ref{-1,2}.
Hence in order to have $\Gamma_{-1,2}\cap \Gamma_{-1,0}>\langle \rho_1 \rangle$, we need $\Gamma_{-1,2}$ to be $D_{12}$ (for otherwise $\Gamma_{-1,2}$ is isomorphic to $D_6$ and therefore, since $\langle\rho_1\rangle$ is maximal  in $\Gamma_{-1,2}$,
$\Gamma_{-1,2} = \Gamma_{-1,2}\cap \Gamma_{-1,0}$ which is clearly impossible).
Now, $\mathcal{G'}_{-1,2}$ has at least one connected component as the following one (the one connected to the ($-1$)-edge in $\mathcal G'$). 
\begin{center}
\begin{tikzpicture}
\vertex[] (1) at (0,0) {};
\vertex[] (2) at (1,0) {};
\vertex[] (3) at (2,0) {};
\node[left=0.3cm of 1,label=left:$-1$] (4) {};
\draw[-] (1) -- (2)
node[pos=0.5,above]{0};
\draw[-] (2) -- (3)
node[pos=0.5,above] {1};
\draw[dashed] (1) -- (4);

\end{tikzpicture}
\end{center}
Moreover, 
$\rho_1$ does not commute with $\rho_0$ nor with $\rho_2$, hence
$\rho_1$ is not in the centre of $\Gamma_{-1,2}$ or $\Gamma_{-1,0}$ so that $\Gamma_{-1,2}\cap \Gamma_{-1,0}$ must contain a rotation common to both dihedral groups: we must have
$(\rho_0\rho_1)^k = (\rho_1\rho_2)^l\neq 1_\Gamma$ for some integers $k$ and $l$.
The orbit above has to be fixed pointwise by $(\rho_0\rho_1)^k$ as $(\rho_1\rho_2)^l$ fixes the leftmost vertex.
Hence $k = 3$ and since $(\rho_0\rho_1)^3\neq 1_\Gamma$, we must also have a 1-edge or a 0-edge as one of the connected components in $\mathcal G'_{-1,2}$ with its vertices swapped by $(\rho_0\rho_1)^3$.
Suppose we have a 0-edge.
It must be connected to the rest of the graph $\mathcal G'$ by an alternating square with 2-edges, in turn connected to the rest of $\mathcal G'$ by a $1$-edge, as represented below.
Hence, the $2$-edges in the square connect the 0-edge to an orbit of size 3 in $\mathcal G'_{-1,2}$, fixed pointwise by $(\rho_0\rho_1)^3 = (\rho_1\rho_2)^l$.

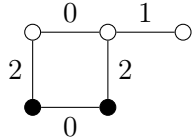
\begin{center}
\begin{tikzpicture}
\vertex[] (1) at (3,1) {};
\vertex[] (2) at (1,1) {};
\vertex[] (3) at (2,1) {};
\vertex[fill=black] (4) at (2,0) {};
\vertex[fill=black] (5) at (1,0) {};
\draw[-] (1) -- (3)
 node[pos=0.5,above] {1};
\draw[-] (3) -- (2)
 node[pos=0.5,above] {0};
\draw[-] (3) -- (4)
 node[pos=0.5,right] {2};
\draw[-] (2) -- (5)
 node[pos=0.5,left] {2};
\draw[-] (5) -- (4)
 node[pos=0.5,below] {0};
\end{tikzpicture}
\end{center}

Now the vertices of the 0-edge have to be fixed pointwise by $(\rho_1\rho_2)^l$ (as it either switches the two leftmost vertices or fixes them) and thus also by $(\rho_0\rho_1)^3$.
So we need a connected component of $\mathcal G'_{-1,2}$ to be a 1-edge $e_1$.
Now, as represented below, $e_1$ must again be connected to the rest of $\mathcal G'$, this time by a path $P$ of alternating 1- and 2-edges, connecting the 1-edge to an orbit $O$ of length 3 of $\mathcal G'_{-1,2}$ that is fixed pointwise by $(\rho_0\rho_1)^3$.

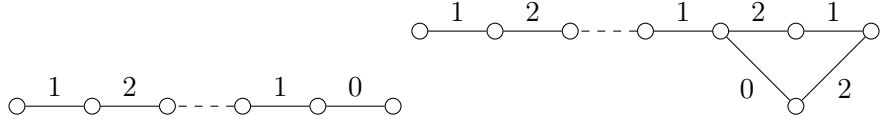
\begin{center}
\begin{tikzpicture}
\vertex[] (1) at (0,1) {};
\vertex[] (2) at (1,1) {};
\vertex[] (3) at (2,1) {};
\vertex[] (4) at (3,1) {};
\vertex[] (5) at (4,1) {};
\vertex[] (6) at (5,1) {};
\draw[-] (1) -- (2)
node[pos=0.5,above]{1};
\draw[-] (2) -- (3)
node[pos=0.5,above] {2};
\draw[dashed] (4) -- (3);
\draw[-] (4) -- (5)
node[pos=0.5,above] {1};
\draw[-] (6) -- (5)
node[pos=0.5,above] {0};
\end{tikzpicture}
\begin{tikzpicture}
\vertex[] (1) at (0,1) {};
\vertex[] (2) at (1,1) {};
\vertex[] (3) at (2,1) {};
\vertex[] (4) at (3,1) {};
\vertex[] (5) at (4,1) {};
\vertex[] (6) at (5,1) {};
\vertex[] (7) at (6,1) {};
\vertex[] (8) at (5,0) {};
\draw[-] (1) -- (2)
node[pos=0.5,above] {1};
\draw[-] (2) -- (3)
node[pos=0.5,above] {2};
\draw[dashed] (4) -- (3);
\draw[-] (4) -- (5)
node[pos=0.5,above] {1};
\draw[-] (6) -- (5)
node[pos=0.5,above] {2};
\draw[-] (6) -- (7)
node[pos=0.5,above] {1};
\draw[-] (8) -- (7)
node[pos=0.5,below,xshift=0.15cm] {2};
\draw[-] (8) -- (5)
node[pos=0.5,below,xshift=-0.15cm] {0};
\end{tikzpicture}
\end{center}

The path $P$ corresponds to an orbit of $\Gamma_{-1,0}$. This orbit has one point fixed by $(\rho_0\rho_1)^3 = (\rho_1\rho_2)^l$, namely the point corresponding to the vertex of $O$ fixed by $\rho_0$.
This forces the vertices of the whole path $P$ to be fixed pointwise by $(\rho_1\rho_2)^k$ and in particular, the vertices of $e_1$ must be fixed pointwise, a contradiction.

This completes the proof that $\Gamma_{-1,0}\cap\Gamma_{-1,2}=\Gamma_{-1,0,2}=\langle \rho_1\rangle$.

\subsection{Proof that $\Gamma_{2}$ is a string C-group}\label{part2}

As a subgroup of $\Gamma$, $\Gamma_{2}$ is a string group generated by  involutions.

\noindent According to Proposition \ref{breakinparts}, it is enough to show that 
\begin{itemize}
    \item $(\Gamma_{-1,2},\{\rho_0,\rho_1\})$
    is a string C-group
    \item $(\Gamma_{1,2},\{\rho_{-1},\rho_0\})$
    is a string C-group
    \item $\Gamma_{-1,2}\cap\Gamma_{1,2}=\Gamma_{-1,1,2}=\langle \rho_0\rangle$
\end{itemize}

The first two points are clear by Lemma \ref{dihedral}.
It remains to show that $\Gamma_{-1,2}\cap\Gamma_{1,2}=\langle \rho_0\rangle$.

As in the previous proof, the subgroup $\Gamma_{-1,2}$ is either $D_6$ or $D_{12}$ by Lemma~\ref{-1,2}.
Hence in order to have $\Gamma_{-1,2}\cap \Gamma_{1,2}>\langle \rho_0 \rangle$, we need $\Gamma_{-1,2}$ to be $D_{12}$ (for otherwise $\Gamma_{-1,2}$ is isomorphic to $D_6$ and therefore, since $\langle\rho_0\rangle$ is maximal  in $\Gamma_{-1,2}$, $\Gamma_{-1,2} = \Gamma_{-1,2}\cap \Gamma_{1,2}$ which is clearly impossible).
Now, $\mathcal{G'}_{-1,2}$ has at least one connected component as the following one (the one connected to the ($-1$)-edge in $\mathcal G'$). 
\begin{center}
\begin{tikzpicture}
\vertex[] (1) at (0,0) {};
\vertex[] (2) at (1,0) {};
\vertex[] (3) at (2,0) {};
\node[left=0.3cm of 1,label=left:$-1$] (4) {};
\draw[-] (1) -- (2)
node[pos=0.5,above]{0};
\draw[-] (2) -- (3)
node[pos=0.5,above] {1};
\draw[dashed] (1) -- (4);
\end{tikzpicture}
\end{center}
Moreover, 
$\rho_0$ does not commute with $\rho_{-1}$ nor with $\rho_1$, hence
$\rho_0$ is not in the centre of $\Gamma_{-1,2}$ or $\Gamma_{1,2}$ so that $\Gamma_{-1,2}\cap\Gamma_{1,2}$ must contain a rotation common to both dihedral groups: we must have
$(\rho_0\rho_1)^k = (\rho_{-1}\rho_0)^l\neq 1_\Gamma$ for some integers $k$ and $l$.
The orbit above has to be fixed pointwise by $(\rho_0\rho_1)^k$ as $(\rho_{-1}\rho_0)^l$ fixes the rightmost vertex.
Hence $k = 3$ and since $(\rho_0\rho_1)^3\neq 1_\Gamma$, we must also have a 1-edge or a 0-edge as one of the connected components in $\mathcal G'_{-1,2}$ with its vertices swapped by $(\rho_0\rho_1)^3$.
Suppose we have a 0-edge.
It must be connected to the rest of the graph $\mathcal G'$ by an alternating square with 2-edges, in turn connected to the rest of $\mathcal G'$ by a $1$-edge, as represented below. Hence the 2-edges in the square connect the 0-edge to an orbit of size 3 in $\mathcal G'_{-1,2}$, fixed pointwise by $(\rho_0\rho_1)^3 = (\rho_{-1}\rho_0)^l$.

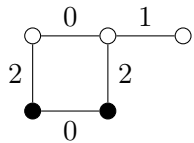
\begin{center}
\begin{tikzpicture}
\vertex[] (1) at (3,1) {};
\vertex[] (2) at (1,1) {};
\vertex[] (3) at (2,1) {};
\vertex[fill=black] (4) at (2,0) {};
\vertex[fill=black] (5) at (1,0) {};
\draw[-] (1) -- (3)
 node[pos=0.5,above] {1};
\draw[-] (3) -- (2)
 node[pos=0.5,above] {0};
\draw[-] (3) -- (4)
 node[pos=0.5,right] {2};
\draw[-] (2) -- (5)
 node[pos=0.5,left] {2};
\draw[-] (5) -- (4)
 node[pos=0.5,below] {0};
\end{tikzpicture}
\end{center}

Hence the vertices of the 0-edge must be fixed pointwise by $(\rho_{-1}\rho_0)^l$ (as it either switches the two leftmost vertices or fixes them) and so also by $(\rho_0\rho_1)^3$.
So we need a connected component of $\mathcal G'_{-1,2}$ to be a 1-edge $e_1$.
The vertices of $e_1$ are fixed pointwise by $(\rho_{-1}\rho_0)^l$, hence they must also be fixed pointwise by $(\rho_0\rho_1)^3$, a contradiction.

This completes the proof that $\Gamma_{-1,2}\cap\Gamma_{1,2}=\Gamma_{-1,1,2}=\langle \rho_0\rangle$.

\subsection{Proof that $\Gamma_{-1} \cap \Gamma_{2}=\Gamma_{-1,2}$}\label{part3}

We first show a lemma needed in the proof.

\begin{lemma} \label{imprim}
    In the settings of Theorem \ref{RAT1} (without taking hypotheses $(1)$ to $(3)$), if $\Gamma_{-1}$ acts imprimitively on one of its orbits $O$ then
     
    \begin{itemize}
    \item any complete system of imprimitivity partitions any triplet of elements forming a $\Gamma_{-1,2}$-orbit of size $3$, and
    
    \item $O$ has another $\Gamma_{-1,2}$-suborbit that is partitioned by the same blocks as $O_{-1}$
    \end{itemize}
where $O_{-1}\subseteq O$ is the $\Gamma_{-1,2}$-orbit represented by the $\mathcal G'_{-1,2}$-connected component adjacent to the $(-1)$-edge.
\end{lemma}

\begin{proof}
Let us denote the subgraph of $\mathcal{G}'$ whose vertices represent $O$ by $\mathcal{G}'_O$.

Note that one cannot have a non-trivial complete system of imprimitivity on less than $4$ elements (since we need at least two blocks containing at least two elements). Hence $\mathcal{G}'_O$ has at least $4$ vertices, three of which represent $O_{-1}$.



Let $S_I$ be a non-trivial complete system of imprimitivity for $\Gamma_{-1}\curvearrowright O$. Since the vertices of $\mathcal{G}'_O$ represent the elements of $O$, throughout, we will assign colours to the vertices in $\mathcal{G}'_O$ to embody the belonging of the corresponding element of $O$ to a given block.


Let us now show that the vertices of any $\mathcal{G}'_{-1,2}$-connected component on $3$ vertices in $\mathcal{G}'_O$ will be partitioned by $S_I$.
Let us consider such a component and give names to its vertices as in the picture below.

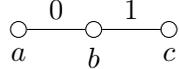
\begin{center}
\begin{tikzpicture}
\vertex[label=below:$a$] (1) at (0,0) {};
\vertex[label=below:$b$] (2) at (1,0) {};
\vertex[label=below:$c$] (3) at (2,0) {};
\draw[-] (1) -- (2)
 node[pos=0.5,above] {0};
\draw[-] (3) -- (2)
 node[pos=0.5,above] {1};
\end{tikzpicture}
\end{center}

Note that for orbit length reasons, there cannot be any $1$-edge connected to $a$ or any $0$-edge connected to $c$. 

If $b$ belongs to a first block, the blue one say, then either $a$ and $c$ are both blue or neither of them is. Indeed, if $a$ (respectively $c$) is blue then $\rho_1$  (resp. $\rho_0$) fixes the blue block so, as the image of $b$ by $\rho_1$ (resp. $\rho_0$), $c$ (resp. $a$) is blue.

Suppose that $a$, $b$ and $c$ are all in the blue block.
This $\mathcal{G}'_{-1,2}$-connected component is not the whole of $\mathcal{G}'_O$ since $\mathcal{G}'_O$ has at least $4$ vertices. The only way to connect these three vertices to the rest of $\mathcal{G}'_O$ is by adding a $2$-edge between one of the three vertices and a new, fourth vertex, say $d$. If $d$ is blue then $\rho_0$, $\rho_1$ and $\rho_2$ all fix the blue block and the whole of $\mathcal{G}'_O$ forms a unique block, a contradiction with the assumption that $S_I$ is non-trivial.
Hence $d$ must belong to a new block, the red one say. Since $\rho_2$ maps elements of the blue block to elements of the red one, $\rho_2$ cannot fix any of our initial three vertices or map them to one another. Hence we must have a $2$-edge in all three, joining each of them to a new one. Since a $2$-edge adjacent to a $0$-edge always leads to an alternating square, we have 

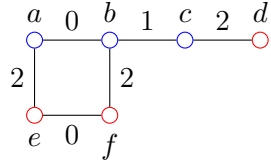
\begin{center}
\begin{tikzpicture}
\vertex[label=above:$a$,color=blue] (1) at (0,1) {};
\vertex[label=above:$b$,color=blue] (2) at (1,1) {};
\vertex[label=above:$c$,color=blue] (3) at (2,1) {};
\vertex[label=below:$e$,color=red] (4) at (0,0) {};
\vertex[label=below:$f$,color=red] (5) at (1,0) {};
\vertex[label=above:$d$,color=red] (6) at (3,1) {};
\draw[-] (1) -- (2)
 node[pos=0.5,above] {0};
\draw[-] (3) -- (2)
 node[pos=0.5,above] {1};
\draw[-] (1) -- (4)
 node[pos=0.5,left] {2};
\draw[-] (5) -- (4)
 node[pos=0.5,below] {0};
\draw[-] (2) -- (5)
 node[pos=0.5,right] {2};
\draw[-] (3) -- (6)
 node[pos=0.5,above] {2};
\end{tikzpicture}
\end{center}

Now, in order to avoid forming a $\mathcal{G}'_{-1,2}$-connected component on more than $3$ vertices, there can only be a $1$-edge incident to one vertex among $e$ and $f$. Thus $\rho_1$ fixes one of $e$ and $f$ and hence the red block.

From this, we can conclude that all vertices of $\mathcal{G}'_O$ are either blue or red. Indeed, we have full knowledge of the action of $\rho_0$, $\rho_1$ and $\rho_2$ on our two blocks.

In particular, the vertex through which $\mathcal{G}'_O$ is connected to the unique $(-1)$-edge must be either red or blue. This leads to a contradiction since no $2$-edge can be adjacent to the $(-1)$-edge and hence this vertex must  be fixed by  $\rho_2$ while $\rho_2$ is not supposed to fix either block.

Hence $a$ and $c$ cannot belong to the same block as $b$. But they also cannot both be in the same new block since, as we have already mentioned, $\rho_1$ fixes $a$ and maps $c$ to an element of the blue block.
Therefore $a$, $b$ and $c$ belong to three different blocks, as required.

This concludes the proof of the first part of this Lemma.


Now, in order to prove the second part, we first show that $\mathcal{G}'_O$ cannot contain any $\mathcal{G}'_{-1,2}$-connected component in the form of a double $0,1$-edge or a single $0$-edge.

Let us start with an observation on the number of leaves and $0,2$-alternating squares connected to exactly one $1$-edge that $\mathcal{G}'_O$ can posses.

It is clear that $\mathcal{G}'_O$ has at least one leaf since one of its vertices has to be incident to our $(-1)$-edge and can thus only have one edge in $\mathcal{G}'_O$, i.e. a $0$-edge (since neither a $1$-edge nor a  $2$-edge can be adjacent to the $(-1)$-edge).

The vertices of $\mathcal{G}'_O$ cannot have a degree higher than $3$ since $\mathcal{G}'_O$ has edges of labels $0$, $1$ and $2$ only. Here the

If we have a vertex with $3$ neighbours then it is adjacent to one edge of each label and to no multiple edge. In  particular, it forces a $0$-edge adjacent to a $2$-edge and thus must be one of the corners of a $0,2$-alternating square.

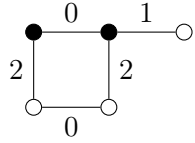
\begin{center}
\begin{tikzpicture}
\vertex[fill=black] (1) at (1,1) {};
\vertex[fill=black] (2) at (2,1) {};
\vertex[] (3) at (3,1) {};
\vertex[] (4) at (1,0) {};
\vertex[] (5) at (2,0) {};
\draw[-] (1) -- (2)
 node[pos=0.5,above] {0};
\draw[-] (3) -- (2)
 node[pos=0.5,above] {1};

\draw[-] (1) -- (4)
 node[pos=0.5,left] {2};
\draw[-] (2) -- (5)
 node[pos=0.5,right] {2};
\draw[-] (5) -- (4)
 node[pos=0.5,below] {0};
\end{tikzpicture}
\end{center}

We cannot have any $1$-edge in the top left vertex as such an edge would be part of a $\mathcal{G}'_{-1,2}$-connected component on more than $3$ vertices.

For the same kind of reason, there can be a $1$-edge in one of the bottom vertices but not in both unless its connecting them. So we have one of the following.

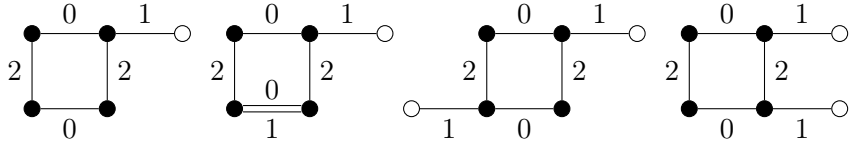
\begin{center}
\begin{tikzpicture}
\vertex[fill=black] (1) at (1,1) {};
\vertex[fill=black] (2) at (2,1) {};
\vertex[] (3) at (3,1) {};
\vertex[fill=black] (4) at (1,0) {};
\vertex[fill=black] (5) at (2,0) {};
\draw[-] (1) -- (2)
 node[pos=0.5,above] {0};
\draw[-] (3) -- (2)
 node[pos=0.5,above] {1};

\draw[-] (1) -- (4)
 node[pos=0.5,left] {2};
\draw[-] (2) -- (5)
 node[pos=0.5,right] {2};
\draw[-] (5) -- (4)
 node[pos=0.5,below] {0};
\end{tikzpicture}
\begin{tikzpicture}
\vertex[fill=black] (1) at (1,1) {};
\vertex[fill=black] (2) at (2,1) {};
\vertex[] (3) at (3,1) {};
\vertex[fill=black] (4) at (1,0) {};
\vertex[fill=black] (5) at (2,0) {};
\draw[-] (1) -- (2)
 node[pos=0.5,above] {0};
\draw[-] (3) -- (2)
 node[pos=0.5,above] {1};

\draw[-] (1) -- (4)
 node[pos=0.5,left] {2};
\draw[-] (2) -- (5)
 node[pos=0.5,right] {2};
\draw[thin,double distance=2pt] (5) -- (4)
 node[pos=0.5,above] {0} node[pos=0.5,below] {1};
\end{tikzpicture}
\begin{tikzpicture}
\vertex[fill=black] (1) at (1,1) {};
\vertex[fill=black] (2) at (2,1) {};
\vertex[] (3) at (3,1) {};
\vertex[fill=black] (4) at (1,0) {};
\vertex[fill=black] (5) at (2,0) {};
\vertex[] (6) at (0,0) {};
\draw[-] (1) -- (2)
 node[pos=0.5,above] {0};
\draw[-] (3) -- (2)
 node[pos=0.5,above] {1};

\draw[-] (1) -- (4)
 node[pos=0.5,left] {2};
\draw[-] (2) -- (5)
 node[pos=0.5,right] {2};
\draw[-] (5) -- (4)
 node[pos=0.5,below] {0};
 \draw[-] (4) -- (6)
 node[pos=0.5,below] {1};
\end{tikzpicture}
\begin{tikzpicture}
\vertex[fill=black] (1) at (1,1) {};
\vertex[fill=black] (2) at (2,1) {};
\vertex[] (3) at (3,1) {};
\vertex[fill=black] (4) at (1,0) {};
\vertex[fill=black] (5) at (2,0) {};`
\vertex[] (6) at (3,0) {};`
\draw[-] (1) -- (2)
 node[pos=0.5,above] {0};
\draw[-] (3) -- (2)
 node[pos=0.5,above] {1};

\draw[-] (1) -- (4)
 node[pos=0.5,left] {2};
\draw[-] (2) -- (5)
 node[pos=0.5,right] {2};
\draw[-] (5) -- (4)
 node[pos=0.5,below] {0};
 \draw[-] (5) -- (6)
 node[pos=0.5,below] {1};
\end{tikzpicture}
\end{center}

In any case, contracting all edges in all $0,2$-alternating squares in $\mathcal{G}'_O$ gives a graph (that is not a valid permutation representation graph for any group anymore) with at least one leaf and vertices with at most $2$ neighbours and hence a path, with exactly two leaves. When we have a square as in one of the first two subgraphs above in $\mathcal{G}'_O$ then one of these two leaves comes from the contraction of a square and not from a leaf in $\mathcal{G}'_O$.

From this, it is clear that $\mathcal{G}'_O$ has at most $2$ leaves and, when it has only one, has exactly one subgraph isomorphic to one of the first two of the four above. This concludes our observation.


Suppose for a contradiction we have a $0,1$-double edge.

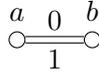
\begin{center}
\begin{tikzpicture}
\vertex[label=above:$a$] (1) at (0,0) {};
\vertex[label=above:$b$] (2) at (1,0) {};
\draw[thin,double distance=2pt] (1) -- (2)
 node[pos=0.5,above] {0} node[pos=0.5,below] {1};
\end{tikzpicture}
\end{center}

The only way to connect this subgraph to the rest of $\mathcal{G}'_O$ is as follows.

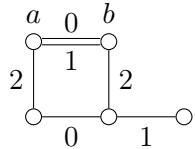
\begin{center}
\begin{tikzpicture}
\vertex[label=above:$a$] (1) at (0,1) {};
\vertex[label=above:$b$] (2) at (1,1) {};
\vertex[] (3) at (0,0) {};
\vertex[] (4) at (1,0) {};
\vertex[] (5) at (2,0) {};
\draw[thin,double distance=2pt] (1) -- (2)
 node[pos=0.5,above] {0} node[pos=0.5,below] {1};
\draw[-] (1) -- (3)
 node[pos=0.5,left] {2};
\draw[-] (2) -- (4)
 node[pos=0.5,right] {2};
\draw[-] (3) -- (4)
 node[pos=0.5,below] {0};
\draw[-] (5) -- (4)
 node[pos=0.5,below] {1};
\end{tikzpicture}
\end{center}

Now, since they form a single $\Gamma_{-1,2}$-orbit, the bottom three vertices must belong to three different blocks, as we've shown above.

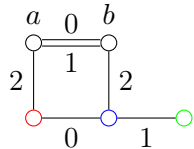
\begin{center}
\begin{tikzpicture}
\vertex[label=above:$a$] (1) at (0,1) {};
\vertex[label=above:$b$] (2) at (1,1) {};
\vertex[color=red] (3) at (0,0) {};
\vertex[color=blue] (4) at (1,0) {};
\vertex[color=green] (5) at (2,0) {};
\draw[thin,double distance=2pt] (1) -- (2)
 node[pos=0.5,above] {0} node[pos=0.5,below] {1};
\draw[-] (1) -- (3)
 node[pos=0.5,left] {2};
\draw[-] (2) -- (4)
 node[pos=0.5,right] {2};
\draw[-] (3) -- (4)
 node[pos=0.5,below] {0};
\draw[-] (5) -- (4)
 node[pos=0.5,below] {1};
\end{tikzpicture}
\end{center}

Now $a$ and $b$ cannot be blue. Indeed, elements of the blue block are mapped to elements of different blocks (i.e. the red and the green one) by $\rho_0$ and $\rho_1$. They also cannot be green (respectively red) since they are mutual images by $\rho_1$ (resp. $\rho_0$) and hence, if one is green (resp. red), the other is blue.
Finally, they must belong to two different blocks because they are mapped on elements of different blocks by $\rho_2$.

Hence we have

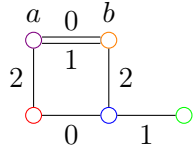
\begin{center}
\begin{tikzpicture}
\vertex[label=above:$a$,color=violet] (1) at (0,1) {};
\vertex[label=above:$b$,color=orange] (2) at (1,1) {};
\vertex[color=red] (3) at (0,0) {};
\vertex[color=blue] (4) at (1,0) {};
\vertex[color=green] (5) at (2,0) {};
\draw[thin,double distance=2pt] (1) -- (2)
 node[pos=0.5,above] {0} node[pos=0.5,below] {1};
\draw[-] (1) -- (3)
 node[pos=0.5,left] {2};
\draw[-] (2) -- (4)
 node[pos=0.5,right] {2};
\draw[-] (3) -- (4)
 node[pos=0.5,below] {0};
\draw[-] (5) -- (4)
 node[pos=0.5,below] {1};
\end{tikzpicture}
\end{center}

For the blocks we have started to form to be non-trivial, we need to have another blue vertex (all blocks must have same size so they are either all trivial or all contain at least two elements). 

Since elements of the blue block are mapped to elements of three different blocks by the three generators, the blue vertex must be of degree $3$.

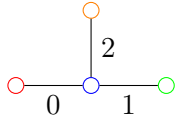
\begin{center}
\begin{tikzpicture}
\vertex[color=orange] (2) at (1,1) {};
\vertex[color=red] (3) at (0,0) {};
\vertex[color=blue] (4) at (1,0) {};
\vertex[color=green] (5) at (2,0) {};
\draw[-] (2) -- (4)
 node[pos=0.5,right] {2};
\draw[-] (3) -- (4)
 node[pos=0.5,below] {0};
\draw[-] (5) -- (4)
 node[pos=0.5,below] {1};
\end{tikzpicture}
\end{center}

Since we have a $0$-edge adjacent to a $2$-edge, we have a $0,2$-alternating square by Lemma \ref{cprlem}. The new vertex must be purple since it represents the image of an orange vertex by $\rho_0$.

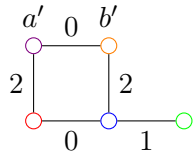
\begin{center}
\begin{tikzpicture}
\vertex[label=above:$a'$,color=violet] (1) at (0,1) {};
\vertex[label=above:$b'$,color=orange] (2) at (1,1) {};
\vertex[color=red] (3) at (0,0) {};
\vertex[color=blue] (4) at (1,0) {};
\vertex[color=green] (5) at (2,0) {};
\draw[-] (1) -- (2)
 node[pos=0.5,above] {0};
\draw[-] (1) -- (3)
 node[pos=0.5,left] {2};
\draw[-] (2) -- (4)
 node[pos=0.5,right] {2};
\draw[-] (3) -- (4)
 node[pos=0.5,below] {0};
\draw[-] (5) -- (4)
 node[pos=0.5,below] {1};
\end{tikzpicture}
\end{center}

For orbit length reasons, there cannot be any $1$-edge in the bottom left vertex. There also cannot be a $1$-edge in both $a'$ and $b'$ unless it is between the two so one of the purple or the orange block is fixed by $\rho_1$. 
This is a contradiction since $\rho_1$ maps the elements of these two blocks to one another. 
Hence there is a double $0,1$-edge between $a'$ and $b'$ and, in particular, no $1$-edge leaves the square in $a'$ or $b'$. We obtain the required contradiction since $\mathcal{G}'_O$ cannot have more than one $0,2$-alternating square connected to the rest by a unique $1$-edge.

With very similar arguments, we can prove that $\mathcal{G}'_O$ does not contain a $\mathcal{G}'_{-1,2}$-connected component consisting of a single $0$-edge. Indeed, such an edge must be part of a $0,2$-alternating square so we find ourselves in a situation very similar to the above.


We have thus shown that the three vertices of the $\mathcal{G}'_{-1,2}$-connected component that is adjacent to our $(-1)$-edge must be partitioned by the blocks in $S_I$. 

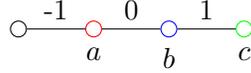
\begin{center}
\begin{tikzpicture}
\vertex[] (1) at (0,0) {};
\vertex[label=below:$a$,color=red] (2) at (1,0) {};
\vertex[label=below:$b$,color=blue] (3) at (2,0) {};
\vertex[label=below:$c$,color=green] (4) at (3,0) {};
\draw[-] (1) -- (2)
 node[pos=0.5,above] {-1};
\draw[-] (3) -- (2)
 node[pos=0.5,above] {0};`
 \draw[-] (3) -- (4)
 node[pos=0.5,above] {1};
\end{tikzpicture}
\end{center}

Since our blocks are non-trivial, the red one must have another element, $a'$ say.

Since $\rho_0$ maps $a$ to $b$, in another block, it must map $a'$ to an element of the blue block as well. Therefore the $a'$ vertex is incident to a $0$-edge. The new vertex, $b'$ say, has to be blue.

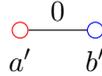
\begin{center}
\begin{tikzpicture}
\vertex[label=below:$a'$,color=red] (2) at (1,0) {};
\vertex[label=below:$b'$,color=blue] (3) at (2,0) {};

\draw[-] (3) -- (2)
 node[pos=0.5,above] {0};
\end{tikzpicture}
\end{center}

Now we have shown that $\mathcal{G}'_O$ has no $\mathcal{G}'_{-1,2}$-connected component consisting of a single $0$-edge or a double $0,1$-edge. Hence there is a $1$-edge in $a'$ or $b'$ (but not both, for orbit length reasons).

If it is in $a'$, we get a contradiction because $\rho_1$ fixes $a$ and hence all elements of the red block so that we have the following.

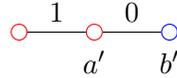
\begin{center}
\begin{tikzpicture}
\vertex[color=red] (1) at (0,0) {};
\vertex[label=below:$a'$,color=red] (2) at (1,0) {};
\vertex[label=below:$b'$,color=blue] (3) at (2,0) {};
\draw[-] (1) -- (2)
 node[pos=0.5,above] {1};
\draw[-] (3) -- (2)
 node[pos=0.5,above] {0};
\end{tikzpicture}
\end{center}

This is impossible by the first part of the Lemma.

Therefore, we have the following subgraph, as required.

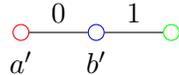
\begin{center}
\begin{tikzpicture}
\vertex[color=green] (1) at (3,0) {};
\vertex[label=below:$a'$,color=red] (2) at (1,0) {};
\vertex[label=below:$b'$,color=blue] (3) at (2,0) {};
\draw[-] (1) -- (3)
 node[pos=0.5,above] {1};
\draw[-] (3) -- (2)
 node[pos=0.5,above] {0};
\end{tikzpicture}
\end{center}

\end{proof}

We now come back to the last part of the proof of Theorem~\ref{RAT1}.
Recall that we need to prove that $\Gamma_{-1,2}=\Gamma_{-1}\cap\Gamma_2$.
It is obvious that $\Gamma_{-1,2}\leq\Gamma_{-1}\cap\Gamma_2$. We have equality if and only if $\Gamma_{-1}\cap\Gamma_2$ acts in the same way as $\Gamma_{-1,2}$ on the $\Gamma_{-1,2}$-orbits.
The possible connected components of $\mathcal{G}_{-1,2}$, whose vertices form these orbits, are, as previously shown, of the following $5$ different shapes.

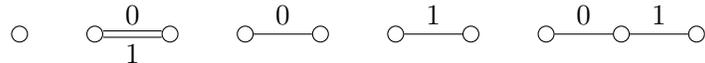
\begin{center}
\begin{tikzpicture}
\vertex[] (1) at (1,0) {};

\vertex[] (2) at (2,0) {};
\vertex[] (3) at (3,0) {};
\draw[thin,double distance=2pt] (2) -- (3)
 node[pos=0.5,above] {0} node[pos=0.5,below] {1};

\vertex[] (4) at (4,0) {};
\vertex[] (5) at (5,0) {};
\draw[-] (4) -- (5)
 node[pos=0.5,above] {0};

\vertex[] (6) at (6,0) {};
\vertex[] (7) at (7,0) {};
\draw[-] (6) -- (7)
 node[pos=0.5,above] {1};

\vertex[] (8) at (8,0) {};
\vertex[] (9) at (9,0) {};
\vertex[] (10) at (10,0) {};
\draw[-] (8) -- (9)
 node[pos=0.5,above] {0};
\draw[-] (9) -- (10)
 node[pos=0.5,above] {1};
\end{tikzpicture}
\end{center}

As in Lemma \ref{-1,2}, the group $\Gamma_{-1,2}$ acts on the vertices of $\mathcal G'$ as a subcartesian product of its transitive constituants.
Moreover, it is isomorphic to a subgroup of  $S_3\times C_2$.

We denote by $\mathcal{G}'_S$ the subgraph of $\mathcal{G}'$ containing all the $\mathcal{G}'_2$-connected components except the one containing the $(-1)$-edge. It is clear that $\mathcal{G}'_S$ is a union of $\mathcal{G}'_{-1,2}$-connected components.

Let us show, first, that $\mathcal{G}'_S$ must contain a connected component on three vertices if it is not either trivial or a union of single $1$-edges. 

No matter what the biggest $\mathcal{G}'_2$-connected component looks like, we must have the following subgraph.

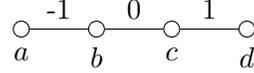
\begin{center}
\begin{tikzpicture}
\vertex[label=below:$a$] (1) at (0,0) {};
\vertex[label=below:$b$] (2) at (1,0) {};
\vertex[label=below:$c$] (3) at (2,0) {};
\vertex[label=below:$d$] (4) at (3,0) {};
\draw[-] (1) -- (2)
 node[pos=0.5,above] {-1};
\draw[-] (3) -- (2)
 node[pos=0.5,above] {0};
\draw[-] (3) -- (4)
 node[pos=0.5,above] {1};
\end{tikzpicture}
\end{center}

Suppose for a contradiction that $\mathcal{G}'_2$ has no connected component on $3$ vertices and is not a union of single $1$-edges and single vertices. Then $\mathcal{G}'_2$ has a connected component in the form of a single $0$-edge or of a double $0,1$-edge. In both cases we have a $0,2$-alternating square as follows.

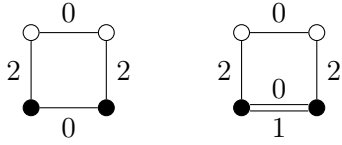
\begin{center}
\begin{tikzpicture}
\vertex[] (2) at (1,1) {};
\vertex[] (3) at (2,1) {};
\vertex[fill=black] (4) at (1,0) {};
\vertex[fill=black] (5) at (2,0) {};
\draw[-] (3) -- (2)
 node[pos=0.5,above] {0};
\draw[-] (2) -- (4)
 node[pos=0.5,left] {2};
\draw[-] (3) -- (5)
 node[pos=0.5,right] {2};
\draw[-] (4) -- (5)
 node[pos=0.5,below] {0};
\end{tikzpicture}
\qquad
\begin{tikzpicture}
\vertex[] (2) at (1,1) {};
\vertex[] (3) at (2,1) {};
\vertex[fill=black] (4) at (1,0) {};
\vertex[fill=black] (5) at (2,0) {};
\draw[-] (3) -- (2)
 node[pos=0.5,above] {0};
\draw[-] (2) -- (4)
 node[pos=0.5,left] {2};
\draw[-] (3) -- (5)
 node[pos=0.5,right] {2};
\draw[thin,double distance=2pt] (4) -- (5)
 node[pos=0.5,above] {0} node[pos=0.5,below] {1};
\end{tikzpicture}
\end{center}

In both cases, this square must be connected to the biggest component, hence there has to be either a $1$-edge or a $(-1)$-edge connected to one of the vertices. A $(-1)$-edge would contradict the string condition. A $1$-edge would contradict the fact that $\mathcal{G}'_2$ has no connected component on 3 vertices.

If $\mathcal{G}'_S$ contains a connected component on three vertices, we denote by $G_S$ the group induced  by $\Gamma_{-1,2}$ on $\mathcal{G}'_S$ according to Lemma \ref{-1,2}. If $\mathcal{G}'_S$ does not contain any connected component on three vertices,  then $G_S\cong C_2$ if it has $1$-edges and $G_S\cong \{1\}$ otherwise.

Suppose first that we are in case (1) of Theorem~\ref{RAT1}.
Then the connected component of $\mathcal{G}'_2$ containing the $(-1)$-edge is  the following.

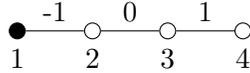
\begin{center}
\begin{tikzpicture}
\vertex[label=below:$1$, fill=black] (1) at (0,0) {};
\vertex[label=below:$2$] (2) at (1,0) {};
\vertex[label=below:$3$] (3) at (2,0) {};
\vertex[label=below:$4$] (4) at (3,0) {};
\draw[-] (1) -- (2)
 node[pos=0.5,above] {-1};
\draw[-] (3) -- (2)
 node[pos=0.5,above] {0};
\draw[-] (3) -- (4)
 node[pos=0.5,above] {1};
\end{tikzpicture}
\end{center}

A quick scan through the possible cases shows that $\Gamma_2$ acts on the $n$ vertices of $\mathcal{G'}$ as $S_4\times G_S$. More precisely, $\Gamma_2$ acts as $S_4$ on the vertices of its orbit of size $4$ and independently as $G_S$ on the rest of its orbits.

We observe that the action of $\Gamma_2$ on the vertices of $\mathcal{G'}$ is only larger than the one of $\Gamma_{-1,2}$ in that $\Gamma_2$ does not fix vertex $1$ and, while it does act on the union of all $\Gamma_{-1,2}$-orbits of size $3$ represented by vertices of $\mathcal{G}'_S$ dependently (that is as $S_3$), it acts on $\{2,3,4\}$ independently.

Since $\Gamma_{-1}\cap\Gamma_2\leq \Gamma_2$, the action of $\Gamma_{-1}\cap\Gamma_2$ cannot be larger than the one of $\Gamma_2$.

Now $\Gamma_{-1}$ fixes vertex $1$ and, by Lemma~\ref{imprim}, there is a $\Gamma_{-1,2}$-orbits of size $3$ represented by vertices of $\mathcal{G}'_S$ on which $\Gamma_{-1}$ does not act independently from $\{2,3,4\}$.

Hence, since $\Gamma_{-1}\cap\Gamma_2\leq\Gamma_{-1}$, the action of $\Gamma_{-1}\cap\Gamma_2$ is the same as the one of $\Gamma_{-1,2}$ on the $n$ vertices of $\mathcal{G}'$. Therefore, we have $\Gamma_{-1}\cap\Gamma_2=\Gamma_{-1,2}$ as required.

Suppose next that we are in case (2) of Theorem~\ref{RAT1}.
Hence the connected component of $\mathcal{G}'_2$ containing the $(-1)$-edge is

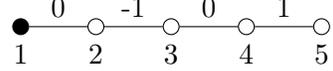
\begin{center}
\begin{tikzpicture}
\vertex[fill=black,label=below:$1$] (1) at (0,0) {};
\vertex[label=below:$2$] (2) at (1,0) {};
\vertex[label=below:$3$] (3) at (2,0) {};
\vertex[label=below:$4$] (4) at (3,0) {};
\vertex[label=below:$5$] (5) at (4,0) {};
\draw[-] (1) -- (2)
 node[pos=0.5,above] {0};
\draw[-] (3) -- (2)
 node[pos=0.5,above] {-1};
\draw[-] (3) -- (4)
 node[pos=0.5,above] {0};
\draw[-] (5) -- (4)
 node[pos=0.5,above] {1};
\end{tikzpicture}
\end{center}

Then $\Gamma_2$ acts on the corresponding orbit as $S_5$ (this can again be shown by hand or using {\sc Magma}) and a quick scan through the possible cases shows that $\Gamma_2$ acts on the $n$ vertices of $\mathcal{G}'$ as $S_5\times G_S$.

The action of $\Gamma_2$ is only larger than the one of $\Gamma_{-1,2}$ in that it can permute $1$ and $2$ independently from any other pair of vertices joined by a $0$- or a $1$-edge in $\mathcal{G}'_S$ and, while it does act on the union of all $\Gamma_{-1,2}$-orbits of size $3$ represented by vertices of $\mathcal{G}'_S$ dependently (that is as $S_3$), it acts on $\{3,4,5\}$ independently. By Lemma \ref{imprim}, $\Gamma_{-1}\cap\Gamma_2\leq\Gamma_{-1}$ does not do the latter. We show that it cannot do the former either. 

This is due to the fact that $\rho_0$, $\rho_1$ and $\rho_2$ are all even. Let $K$ be the permutation group whose permutation representation graph is $\mathcal{G}'$ deprived from the $0$-edge between vertices $1$ and $2$ and the $(-1)$-edge (between $2$ and $3$). By our previous remark about the action of $\Gamma_{-1}$ on the $\Gamma_{-1,2}$-orbits of size $3$, $K\cong G_S$. 
We note that $\Gamma_{-1}$ is a sesqui-extension of $K$. By Proposition \ref{sesqui},  either it contains the $(1,2)$ transposition or it is isomorphic to $K$. Since $\rho_0$, $\rho_1$ and $\rho_2$ are all even,  $\Gamma_{-1}\leq A_n$ and it contains no transposition. Hence it is isomorphic to $K\cong G_S$, as required.

 Hence, since $\Gamma_{-1}\cap\Gamma_2\leq\Gamma_{-1}$, the action of $\Gamma_{-1}\cap\Gamma_2$ is the same as the one of $\Gamma_{-1,2}$ on the $n$ vertices of $\mathcal{G}'$. Therefore, we have $\Gamma_{-1}\cap\Gamma_2=\Gamma_{-1,2}$ as required.

Finally, suppose that we are in case (3) of Theorem~\ref{RAT1}.
Hence the largest $\mathcal{G}'_2$-connected component has $6$ vertices.

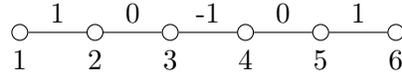
\begin{center}
\begin{tikzpicture}
\vertex[label=below:$1$] (1) at (0,0) {};
\vertex[label=below:$2$] (2) at (1,0) {};
\vertex[label=below:$3$] (3) at (2,0) {};
\vertex[label=below:$4$] (4) at (3,0) {};
\vertex[label=below:$5$] (5) at (4,0) {};
\vertex[label=below:$6$] (6) at (5,0) {};
\draw[-] (1) -- (2)
 node[pos=0.5,above] {1};
\draw[-] (3) -- (2)
 node[pos=0.5,above] {0};
\draw[-] (3) -- (4)
 node[pos=0.5,above] {-1};
\draw[-] (5) -- (4)
 node[pos=0.5,above] {0};
\draw[-] (5) -- (6)
 node[pos=0.5,above] {1};
\end{tikzpicture}
\end{center}

It is rather easy to show (by hand or using {\sc Magma}) that $\Gamma_2$ acts on the corresponding orbit as the symmetry group of the cube, i.e. $C_2\times S_4$, where the elements of the orbit can be viewed as the six faces of the said cube. In particular, $\Gamma_2$ does not act independently on $\{1,2,3\}$ and $\{4,5,6\}$. Hence the action of $\Gamma_{-1}\cap\Gamma_2$ is the same as the one of $\Gamma_{-1,2}$ on the $n$ vertices of $\mathcal{G}'$. Therefore, we have $\Gamma_{-1}\cap\Gamma_2=\Gamma_{-1,2}$ as required.

This, combined with Section~\ref{part1} and Section~\ref{part2}, concludes the proof of Theorem~\ref{RAT1}.


\section{Conclusion}\label{conclusion}

The proof of Theorem~\ref{RAT1} shows that reversing the Rank Reduction Theorem is highly non-trivial.

In order to motivate the extra conditions $(1)$ and $(2)$ of Theorem~\ref{RAT1}, we give here examples where the rank augmentation fails. All non-trivial computations were carried out using {\sc Magma}.

Let us consider the following CPR graph.

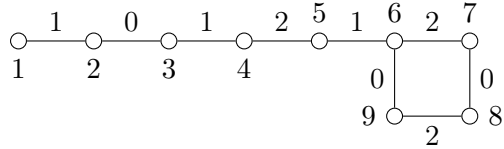
\begin{center}
\begin{tikzpicture}
\vertex[label=below:$1$] (1) at (0,1) {};
\vertex[label=below:$2$] (2) at (1,1) {};
\vertex[label=below:$3$] (3) at (2,1) {};
\vertex[label=below:$4$] (4) at (3,1) {};
\vertex[label=above:$5$] (5) at (4,1) {};
\vertex[label=above:$6$] (6) at (5,1) {};
\vertex[label=above:$7$] (7) at (6,1) {};
\vertex[label=right:$8$] (8) at (6,0) {};
\vertex[label=left:$9$] (9) at (5,0) {};
\draw[-] (1) -- (2)
 node[pos=0.5,above] {1};
\draw[-] (3) -- (2)
 node[pos=0.5,above] {0};
\draw[-] (3) -- (4)
 node[pos=0.5,above] {1};
\draw[-] (5) -- (4)
 node[pos=0.5,above] {2};
\draw[-] (5) -- (6)
 node[pos=0.5,above] {1};
\draw[-] (7) -- (6)
 node[pos=0.5,above] {2};
\draw[-] (7) -- (8)
 node[pos=0.5,right] {0};
\draw[-] (9) -- (6)
 node[pos=0.5,left] {0};
\draw[-] (9) -- (8)
 node[pos=0.5,below] {2};
\end{tikzpicture}
\end{center}

It is a permutation representation graph for  the string C-group $G:=\langle(2,3)(6,9)(7,8),(1,2)(3,4)(5,6),(4,5)(6,7)(8,9)\rangle\cong S_9$.

If we augment the rank of this string C-group by replacing the leftmost $1$-edge in its CPR graph by a $(-1)$-edge, we obtain a group $\Gamma$ with the following permutation graph.

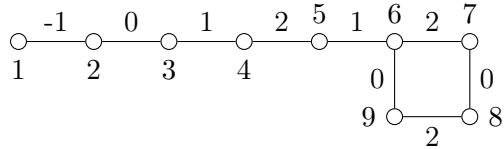
\begin{center}
\begin{tikzpicture}
\vertex[label=below:$1$] (1) at (0,1) {};
\vertex[label=below:$2$] (2) at (1,1) {};
\vertex[label=below:$3$] (3) at (2,1) {};
\vertex[label=below:$4$] (4) at (3,1) {};
\vertex[label=above:$5$] (5) at (4,1) {};
\vertex[label=above:$6$] (6) at (5,1) {};
\vertex[label=above:$7$] (7) at (6,1) {};
\vertex[label=right:$8$] (8) at (6,0) {};
\vertex[label=left:$9$] (9) at (5,0) {};
\draw[-] (1) -- (2)
 node[pos=0.5,above] {-1};
\draw[-] (3) -- (2)
 node[pos=0.5,above] {0};
\draw[-] (3) -- (4)
 node[pos=0.5,above] {1};
\draw[-] (5) -- (4)
 node[pos=0.5,above] {2};
\draw[-] (5) -- (6)
 node[pos=0.5,above] {1};
\draw[-] (7) -- (6)
 node[pos=0.5,above] {2};
\draw[-] (7) -- (8)
 node[pos=0.5,right] {0};
\draw[-] (9) -- (6)
 node[pos=0.5,left] {0};
\draw[-] (9) -- (8)
 node[pos=0.5,below] {2};
\end{tikzpicture}
\end{center}

Note that the action of $\Gamma_{-1}$ on its biggest (sole non-trivial) orbit (represented by the eight rightmost vertices of the above graph) is primitive (in fact $\Gamma_{-1}\cong S_8$). 

Note also that $\Gamma_2$, as in all examples corresponding to case $(1)$ of Theorem \ref{RAT1}, acts independently on the $4$ leftmost vertices, as $S_4$.

This new sggi does not satisfy the intersection property since $\Gamma_{-1}\cap\Gamma_2$ acts independently on $\{2,3,4\}$ and $\{5,6,9\}$ while $\Gamma_{-1,2}$ does not.

Now let us look at case $(2)$ of Theorem \ref{RAT1} and consider the following CPR graph.

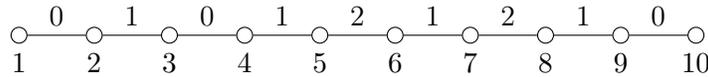
\begin{center}
\begin{tikzpicture}
\vertex[label=below:$1$] (1) at (0,1) {};
\vertex[label=below:$2$] (2) at (1,1) {};
\vertex[label=below:$3$] (3) at (2,1) {};
\vertex[label=below:$4$] (4) at (3,1) {};
\vertex[label=below:$5$] (5) at (4,1) {};
\vertex[label=below:$6$] (6) at (5,1) {};
\vertex[label=below:$7$] (7) at (6,1) {};
\vertex[label=below:$8$] (8) at (7,1) {};
\vertex[label=below:$9$] (9) at (8,1) {};
\vertex[label=below:$10$] (10) at (9,1) {};
\draw[-] (1) -- (2)
 node[pos=0.5,above] {0};
\draw[-] (3) -- (2)
 node[pos=0.5,above] {1};
\draw[-] (3) -- (4)
 node[pos=0.5,above] {0};
\draw[-] (5) -- (4)
 node[pos=0.5,above] {1};
\draw[-] (5) -- (6)
 node[pos=0.5,above] {2};
\draw[-] (7) -- (6)
 node[pos=0.5,above] {1};
\draw[-] (7) -- (8)
 node[pos=0.5,above] {2};
\draw[-] (9) -- (8)
 node[pos=0.5,above] {1};
\draw[-] (9) -- (10)
 node[pos=0.5,above] {0};
\end{tikzpicture}
\end{center}

It is the permutation representation graph for the string C-group $G:=\langle(1,2)(3,4)(9,10),(2,3)(4,5)(6,7)(8,9),(5,6)(7,8)\rangle\cong S_{10}$.

If we augment the rank of this string C-group by replacing the leftmost $1$-edge in its CPR graph by a $(-1)$-edge, we obtain a group $\Gamma$ with the following permutation graph.

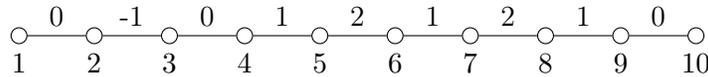
\begin{center}
\begin{tikzpicture}
\vertex[label=below:$1$] (1) at (0,1) {};
\vertex[label=below:$2$] (2) at (1,1) {};
\vertex[label=below:$3$] (3) at (2,1) {};
\vertex[label=below:$4$] (4) at (3,1) {};
\vertex[label=below:$5$] (5) at (4,1) {};
\vertex[label=below:$6$] (6) at (5,1) {};
\vertex[label=below:$7$] (7) at (6,1) {};
\vertex[label=below:$8$] (8) at (7,1) {};
\vertex[label=below:$9$] (9) at (8,1) {};
\vertex[label=below:$10$] (10) at (9,1) {};
\draw[-] (1) -- (2)
 node[pos=0.5,above] {0};
\draw[-] (3) -- (2)
 node[pos=0.5,above] {-1};
\draw[-] (3) -- (4)
 node[pos=0.5,above] {0};
\draw[-] (5) -- (4)
 node[pos=0.5,above] {1};
\draw[-] (5) -- (6)
 node[pos=0.5,above] {2};
\draw[-] (7) -- (6)
 node[pos=0.5,above] {1};
\draw[-] (7) -- (8)
 node[pos=0.5,above] {2};
\draw[-] (9) -- (8)
 node[pos=0.5,above] {1};
\draw[-] (9) -- (10)
 node[pos=0.5,above] {0};
\end{tikzpicture}
\end{center}

Note that the action of $\Gamma_{-1}$ on its biggest orbit (represented by the eight rightmost vertices of the above graph) is imprimitive (with blocks $\{3,10\}$, $\{4,9\}$, $\{5,8\}$, $\{6,7\}$) but neither $\rho_0$ nor $\rho_1$ is even. In fact, $\Gamma_{-1}$ is a $0$-sesqui-extension of $\langle(3,4)(9,10),(4,5)(5,7)(8,9),(5,6)(7,8)\rangle$ (denoted by $K$  in Section \ref{part3}) with respect to $(1,2)$ that contains the transposition $(1,2)$.

Note also that $\Gamma_2$ acts as $S_5$ on the five leftmost vertices of the permutation graph (as it will do in all examples of case $(2)$ of Theorem \ref{RAT1}).

We can conclude that $\Gamma$ does not satisfy the intersection property since $(1,2)$ belongs to $\Gamma_{-1}\cap\Gamma_2$ but not to $\Gamma_{-1,2}$.

In this last example, we have the parity of $\rho_0,\rho_1$ and $\rho_2$ but not the imprimitivity hypothesis.

Let us consider the following CPR graph.

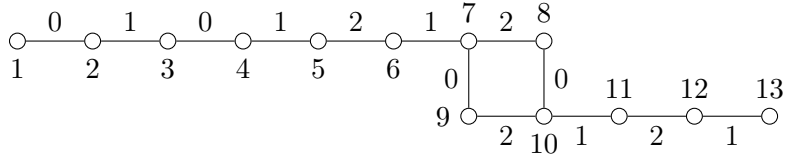
\begin{center}
\begin{tikzpicture}
\vertex[label=below:$1$] (1) at (0,1) {};
\vertex[label=below:$2$] (2) at (1,1) {};
\vertex[label=below:$3$] (3) at (2,1) {};
\vertex[label=below:$4$] (4) at (3,1) {};
\vertex[label=below:$5$] (5) at (4,1) {};
\vertex[label=below:$6$] (6) at (5,1) {};
\vertex[label=above:$7$] (7) at (6,1) {};
\vertex[label=above:$8$] (8) at (7,1) {};
\vertex[label=left:$9$] (9) at (6,0) {};
\vertex[label=below:$10$] (10) at (7,0) {};
\vertex[label=above:$11$] (11) at (8,0) {};
\vertex[label=above:$12$] (12) at (9,0) {};
\vertex[label=above:$13$] (13) at (10,0) {};
\draw[-] (1) -- (2)
 node[pos=0.5,above] {0};
\draw[-] (3) -- (2)
 node[pos=0.5,above] {1};
\draw[-] (3) -- (4)
 node[pos=0.5,above] {0};
\draw[-] (5) -- (4)
 node[pos=0.5,above] {1};
\draw[-] (5) -- (6)
 node[pos=0.5,above] {2};
\draw[-] (7) -- (6)
 node[pos=0.5,above] {1};
\draw[-] (7) -- (8)
 node[pos=0.5,above] {2};
\draw[-] (7) -- (9)
 node[pos=0.5,left] {0};
\draw[-] (8) -- (10)
 node[pos=0.5,right] {0};
\draw[-] (10) -- (9)
 node[pos=0.5,below] {2};
\draw[-] (10) -- (11)
 node[pos=0.5,below] {1};
\draw[-] (11) -- (12)
 node[pos=0.5,below] {2};
\draw[-] (13) -- (12)
 node[pos=0.5,below] {1};
\end{tikzpicture}
\end{center}

It is the permutation representation graph for a string C-group $G:=\langle(1,2)(3,4)(7,9)(8,10),(2,3)(4,5)(6,7)(10,11)(12,13),(5,6)(7,8)(9,10)(11,12)\rangle\cong S_{13}$.

If we augment the rank of this string C-group by replacing the leftmost $1$-edge in its CPR graph by a $(-1)$-edge,  we obtain a group $\Gamma$ with the following permutation graph.

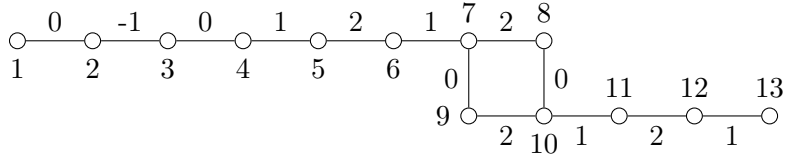
\begin{center}
\begin{tikzpicture}
\vertex[label=below:$1$] (1) at (0,1) {};
\vertex[label=below:$2$] (2) at (1,1) {};
\vertex[label=below:$3$] (3) at (2,1) {};
\vertex[label=below:$4$] (4) at (3,1) {};
\vertex[label=below:$5$] (5) at (4,1) {};
\vertex[label=below:$6$] (6) at (5,1) {};
\vertex[label=above:$7$] (7) at (6,1) {};
\vertex[label=above:$8$] (8) at (7,1) {};
\vertex[label=left:$9$] (9) at (6,0) {};
\vertex[label=below:$10$] (10) at (7,0) {};
\vertex[label=above:$11$] (11) at (8,0) {};
\vertex[label=above:$12$] (12) at (9,0) {};
\vertex[label=above:$13$] (13) at (10,0) {};
\draw[-] (1) -- (2)
 node[pos=0.5,above] {0};
\draw[-] (3) -- (2)
 node[pos=0.5,above] {-1};
\draw[-] (3) -- (4)
 node[pos=0.5,above] {0};
\draw[-] (5) -- (4)
 node[pos=0.5,above] {1};
\draw[-] (5) -- (6)
 node[pos=0.5,above] {2};
\draw[-] (7) -- (6)
 node[pos=0.5,above] {1};
\draw[-] (7) -- (8)
 node[pos=0.5,above] {2};
\draw[-] (7) -- (9)
 node[pos=0.5,left] {0};
\draw[-] (8) -- (10)
 node[pos=0.5,right] {0};
\draw[-] (10) -- (9)
 node[pos=0.5,below] {2};
\draw[-] (10) -- (11)
 node[pos=0.5,below] {1};
\draw[-] (11) -- (12)
 node[pos=0.5,below] {2};
\draw[-] (13) -- (12)
 node[pos=0.5,below] {1};
\end{tikzpicture}
\end{center}

Note that $\rho_0$, $\rho_1$ and$\rho_2$ are even permutations but the action of $\Gamma_{-1}$ on its biggest orbit (represented by the eleven rightmost vertices of the above graph) is primitive even though it contains more than one $\Gamma_{-1,2}$-orbit of size $3$. In fact, $\Gamma_{-1}$ acts as $S_{11}$ on the aforementioned orbit.

Note also that, as mentioned in the previous example, $\Gamma_2$ acts as $S_5$ on the five leftmost vertices.

We can conclude that $\Gamma$ does not satisfy the intersection property since $\Gamma_{-1}$ and $\Gamma_2$ act independently on $\{3,4,5\}$ and, say, $\{6,7,9\}$ (but also $\{8,10,11\}$) while $\Gamma_{-1,2}$ does not.



The hypothesis on the length of the orbits of $\Gamma_{-1,2}$ in Theorem~\ref{RAT1} implies that the Schl\"afli type of the string C-group representation we get with the rank augmentation theorem starts with a 3 or a 6.
A natural extension of this work would be to see what can be said if we remove that hypothesis. Now $\Gamma_{-1,2}=\langle\rho_0,\rho_1\rangle$ is always a dihedral group $D_{2p_1}$ where $p_1$ is the order of $(\rho_0\rho_1)$.
When we have $\Gamma_{-1,2}$-orbits of size no larger than $3$, the action of $\Gamma_{-1}$ and $\Gamma_2$ on these orbits can never exceed $S_3\cong D_6$.
When one allows larger orbits, their action on these can reach $S_m$ (where $m$ is the size of the said orbit), which is not isomorphic to  a dihedral group. Hence, in the latter case, we risk not having $\Gamma_{-1}\cap\Gamma_2=\Gamma_{-1,2}$.

Another natural extension of this work would be to take string C-group representations of higher ranks and try to augment their rank.


\section{Acknowledgements}

The authors gratefully acknowledge financial support from the Actions de Recherche Concertées of the Communauté Française Wallonie Bruxelles. 

\bibliographystyle{plain}
\bibliography{mybibliography}

\begin{thebibliography}{1}

\bibitem{RRT}
Peter Brooksbank and Dimitri Leemans.
\newblock Rank reduction of string {C}-group representations.
\newblock {\em Proc. Amre. Math. Soc.}, 147(12):5421--5426, 2019.

\bibitem{CFL2022}
Peter~J. Cameron, Maria~Elisa Fernandes, and Dimitri Leemans.
\newblock The number of string {C}-groups of high rank.
\newblock Preprint, 2022.

\bibitem{altn}
Peter~J. Cameron, Maria~Elisa Fernandes, Dimitri Leemans, and Mark Mixer.
\newblock Highest rank of a polytope for {$A_n$}.
\newblock {\em Proc. Lond. Math. Soc. (3)}, 115(1):135--176, 2017.

\bibitem{PolSymGroups}
Maria~Elisa Fernandes and Dimitri Leemans.
\newblock Polytopes of high rank for the symmetric groups.
\newblock {\em Adv. Math.}, 228(6):3207--3222, 2011.

\bibitem{altnallranks}
Maria~Elisa Fernandes and Dimitri Leemans.
\newblock String {C}-group representations of alternating groups.
\newblock {\em Ars Math. Contemp.}, 17(1):291--310, 2019.

\bibitem{RankofPolAltGroups}
Maria~Elisa Fernandes, Dimitri Leemans, and Mark Mixer.
\newblock All alternating groups ${A}_n$ with $n\geq 12$ have polytopes of rank
  $\lfloor (n-1)/2\rfloor$.
\newblock {\em SIAM J. Discrete Math.}, 26(2):482--498, 2012.

\bibitem{AtlasSmallGroups}
Dimitri Leemans and Laurence Vauthier.
\newblock An atlas of abstract regular polytopes for small groups.
\newblock {\em Aequationes Math.}, 72(3):313--320, 2006.

\bibitem{ARP}
Peter McMullen and Egon Schulte.
\newblock {\em Abstract regular polytopes}, volume~92.
\newblock Cambridge University Press, 2002.

\bibitem{cprbible}
Daniel Pellicer.
\newblock {CPR} graphs and regular polytopes.
\newblock {\em European J. Combin.}, 29(1):59--71, 2008.

\end{thebibliography}
\end{document}